\documentclass[number]{elsarticle}
\usepackage{geometry}                
\usepackage{graphicx}
\usepackage{amssymb}
\usepackage{epstopdf}
\usepackage{amsthm}
\usepackage{amsmath}
\newtheorem{thm}{Theorem}

\newtheorem{prop}[thm]{Proposition}
\newtheorem{cor}[thm]{Corollary}
\newdefinition{rmk}{Remark}
\newdefinition{Def}{Definition}
\newdefinition{ex}{Example}
\newcommand{\<}{\langle}
\renewcommand{\>}{\rangle}
\newcommand{\E}{\mathbb E}
\begin{document}
\title{White Noise Representation of Gaussian Random Fields}
\author{Zachary Gelbaum}
\ead {gelbaumz@math.oregonstate.edu}
\address{Department of Mathematics\\ Oregon State University\\ Corvallis, Oregon 97331-4605, USA}
\date{}                                           

\begin{abstract}
	
	We obtain a representation theorem for Banach space valued Gaussian random variables as integrals against a white noise.  As a corollary we obtain necessary and sufficient conditions for the existence of a white noise representation for a Gaussian random field indexed by a compact measure space.   As an application we show how existing theory for integration with respect to Gaussian processes indexed by $[0,1]$ can be extended to Gaussian fields indexed by compact measure spaces.
	
\end{abstract}
\maketitle
\section{Introduction}

Much of literature regarding the representation of Gaussian processes as integrals against white noise has focused on processes indexed by $\mathbb R$, in particular canonical representations (most recently see \cite{MR2599216} and references therein) and Volterra processes (e.g$.$ \cite{MR1849177, MR1999794}).  An example of the use of such integral representations is the construction of a stochastic calculus for Gaussian processes admitting a white noise representation with a Volterra kernel (e.g. \cite{MR1849177, MR1972287}).  

In this paper we study white noise representations for Gaussian random variables in Banach spaces, focusing in particular on Gaussian random fields indexed by a compact measure space.  We show that the existence of a representation as an integral against a white noise on a Hilbert space $H$ is equivalent to the existence of a version of the field whose sample paths lie almost surely in $H$.  For example as a consequence of our results a centered Gaussian process $Y_t$ indexed by $[0,1]$ admits a representation \[Y_t\stackrel{d}{=}\int_0^1h(t,z)dW(z)\] for some $h\in L^2([0,1]\times[0,1],d\nu\times d\nu)$, $\nu$ a measure on $[0,1]$ and $W$ the white noise on $L^2([0,1],d\nu)$ if and only if there is a version of $Y_t$ whose sample paths belong almost surely to $L^2([0,1],d\nu)$.  

The stochastic integral for Volterra processes developed in \cite{MR1972287} depends on the existence of a white noise integral representation for the integrator.  If there exists an integral representation for a given Gaussian field then the method in \cite{MR1972287} can be extended to define a stochastic integral with respect to this field.  We describe this extension for Gaussian random fields indexed by a compact measure space whose sample paths are almost surely square integrable.  

Section 2 contains preliminaries we will need from Malliavin Calculus and the theory of Gaussian measures over Banach spaces.  In section 3, Theorem 1 gives our abstract representation theorem and Corollary 2 specializes to Gaussian random fields indexed by a compact measure space.  Section 4 contains the extension of results in \cite{MR1972287}.

\section{Preliminaries}   
\subsection{Malliavin Calculus}
We collect here only those parts of the theory that we will explicitly use, see \cite{MR2200233}.  

\begin{Def}Suppose we have a Hilbert space $H$.  Then there exists a complete probability space $(\Omega,\mathcal F,\mathbb P)$ and a map $W:H\to L^2(\Omega,\mathbb P)$ satisfying the following:\begin{enumerate}\item $W(h)$ is a centered Gaussian random variable with $E[W(h)^2]=\|h\|_H$\item $E[W(h_1)W(h_2)]=\<h_1,h_2\>_H$
\end{enumerate}
This process is unique up to distribution and is called the {\em{Isonormal}} or {\em White Noise Process} on $H$.
\end{Def}

The classical example is $H=L^2[0,1]$ and $W(h)$ is the Wiener-Ito integral of $h\in L^2$.

Let $\mathcal S$ denote the set of random variables of the form \[F=f(W(h_1),...,W(h_n))\] for some $f\in C^\infty(\mathbb R^n)$ such that $f$ and all its derivatives have at most polynomial growth at infinity.  For $F\in\mathcal S$ we define the derivative as \[DF=\sum_1^n\partial_jf(W(h_1),...,W(h_n))h_j.\]  We denote by $\mathbb D$ the closure of $\mathcal S$ with respect to the norm induced by the inner product \[\langle F,G\rangle_{D}=E[FG]+E[\langle DF,DG\rangle_H].\]($\mathbb D$ is usually denoted $\mathbb D^{1,2}$.)

We also define a directional derivative for $h\in H$ as \[D_hF=\<DF,h\>_H.\]

$D$ is then a closed operator from $L^2(\Omega)$ to $L^2(\Omega,H)$ and $dom(D)=\mathbb D$.  Further, $\mathbb D$ is dense in $L^2(\Omega)$.  Thus we can speak of the adjoint of $D$ as an operator from $L^2(\Omega,H)$ to $L^2(\Omega)$.  This operator is called the divergence operator and denoted by $\delta$.

$dom(\delta)$ is the set of all $u\in L^2(\Omega,H)$ such that there exists a constant $c$ (depending on $u$) with \[|\E[\< DF,u\>_H]|\leq c\|F\|\] for all $F\in \mathbb D$.  For $u\in dom(\delta)$ $\delta(u)$ is characterized by \[\E[F\delta(u)]=\E[\<DF,u\>_H]\] for all $F\in\mathbb D$.

For examples and descriptions of the domain of $\delta$ see \cite{MR2200233}, section 1.3.1.

When we want to specify the isonormal process defining the divergence we write $\delta^W$.  We will also use the following notations interchangeably \[\delta^W(u)\,,\,\int udW.\]
\subsection{Gaussian Measures on Banach Spaces}
Here we collect the necessary facts regarding Gaussian measures on Banach spaces and related notions that we will use in what follows.  For proofs and further details see e.g. \cite{MR1642391,  MR2235463, MR2244975, MR0461643, MR1435288}.  All Banach spaces are assumed real and separable throughout.

\begin{Def}Let $B$ be a Banach space.  A probability measure $\mu$ on the borel sigma field $\mathcal B$ of $B$ is called Gaussian if for every $l\in B^*$ the random variable $l(x):(B,\mathcal{B},\mu)\to\mathbb R$ is Gaussian.  The mean of $\mu$ is defined as \[m(\mu)=\int_B xd\mu(x).\]  $\mu$ is called centered if $m(\mu)=0$.  The (topological) support of $\mu$ in $B$, denoted $B_0$, is defined as the smallest closed subspace of $B$ with $\mu$-measure equal to $1$.
\end{Def}

The mean of a Guassian measure is always an element of $B$, and thus it suffices to consider only centered Gaussian measures as we can then acquire any Gaussian measure via a simple translation of a centered one.  For the remainder of the paper all measures considered are centered.
\begin{Def} The covariance of a Gaussian measure is the bilinear form $C_\mu:B^*\times B^*\to\mathbb R$ given by \[C_\mu(k,l)=\mathbb E[k(X)l(X)]=\int_Bk(x)l(x)d\mu(x).\]
\end{Def}

Any gaussian measure is completely determined by its covariance: if for two Gaussian measures $\mu$, $\nu$ on $B$ we have $C_\mu=C_\nu$ on $B^*\times B^*$ then $\mu=\nu$.  

If $H$ is a Hilbert space then \[C_\mu(f,g)=\E[\<X,f\>\<X,g\>]=\int_B\<x,f\>\<x,g\>d\mu(x)\] defines a continuous, positive, symmetric bilinear form on $H\times H$ and thus determines a positive symmetric operator $K_\mu$ on $H$.  $K_\mu$ is of trace class and is injective if and only if $\mu(H)=1$.  Conversely, any positive trace class operator on $H$ uniquely determines a Guassian measure on $H$ \cite{MR2244975}.  Whenever we consider a Gaussian measure $\mu$ over a Hilbert space $H$ we can after restriction to a closed subspace assume $\mu(H)=1$ and do so throughout.
	
	We will denote by $H_\mu$ the Reproducing Kernel Hilbert Space (RKHS) associated to a Gaussian measure $\mu$ on $B$ .  There are various equivalent constructions of the RKHS.  We follow \cite{MR1435288} and refer the interested reader there for complete details.
	
	For any fixed $l\in B^*$, $C_\mu(l,\cdot)\in B$ (this is a non trivial result in the theory).  Consider the linear span of these functions, \[S=span\{C_\mu(l,\cdot)\ :\ l\in B^*\}.\]
Define an inner product on $S$ as follows: if $\phi(\cdot)=\sum_1^na_iC_\mu(l_i,\cdot)$ and $\psi(\cdot)=\sum_1^mb_jC_\mu(k_j,\cdot)$ then \[<\phi,\psi>_{H_\mu}\equiv \sum_1^n\sum_1^ma_ib_jC_\mu(l_i,k_j).\]
$H_\mu$ is defined to be the closure of $S$ under the associated norm $\|\cdot\|_{H_\mu}$.   This norm is stronger than $\|\cdot\|_B$, $H_\mu$ is a dense subset of $B_0$ and $H_\mu$ has the reproducing property with reproducing kernel $C_\mu(l,k)$:  \[\<\phi(\cdot),C_\mu(l,\cdot)\>_{H_\mu}=\phi(l) \qquad \forall\, l\in B^*,\, \phi\in H_\mu.\]
\begin{rmk} Often one begins with a collection of random variables indexed by some set, $\{Y_t\}_{t\in T}$.  For example suppose $(T,\nu)$ is a finite measure space.  Then setting $K(s,t)=\E[Y_sY_t]$ and supposing that application of Fubini-Tonelli is justified we have for $f,g\in L^2(T)$ \[\E[\<Y,f\>\<Y,g\>]=\int_T\int_T\E[Y_s,Y_t]f(s)g(t)d\nu d\nu=\<K(s,t)(f),g\>\] where we denote $\int_TK(s,t)f(s)d\nu(s)$ by $K(s,t)(f)$.  If one verifies that this last operator is positive symmetric and trace class then the above collection $\{Y_t\}_{t\in T}$ determines a measure $\mu$ on $L^2(T)$ and the above construction goes through with $C_\mu(f,g)=\<K(s,t)(f),g\>$ and the end result is the same with $H_\mu$ a space of functions over $T$.

\end{rmk}

Define $H_X$ to be the closed linear span of $\{X(l)\}_{l\in B^*}$ in $L^2(\Omega,\mathbb P)$ with inner product \hfill\\$\<X(l),X(l')\>_{H_X}=C_\mu(l,l')$ (again for simplicity assume $X$ is nondegenerate).  From the reproducing property we can define a mapping $R_X$ from $H_\mu$ to $H_X$ given initially on $S$ by \[R_X(\sum_1^n c_kC_\mu(l_k,\cdot))=\sum_1^k c_kX(l)\] and extending to an isometry.  This isometry defines the isonormal process on $H_\mu$.

In the case that $H$ is a Hilbert space and $\mu$ a Gaussian measure on $H$ with covariance operator $K$ it is known that $H_\mu=\sqrt{K}(H)$ with inner product $\<\sqrt{K}(x),\sqrt{K}(y)\>_{H_\mu}=\<x,y\>_H$.

It was shown in \cite{MR0260010} that given a Banach space $B$ there exists a Hilbert space $H$ such that $B$ is continuously embedded as a dense subset of $H$.  Any Gaussian measure $\mu$ on $B$ uniquely extends to a Gaussian measure $\mu_H$ on $H$.   The converse question of whether a given Gaussian measure on $H$ restricts to a Gaussian measure on $B$ is far more delicate.  There are some known conditions e.g$.$ \cite{MR0267614}.  The particular case when $X$ is a metric space and $B=C(X)$ has been the subject of extensive research \cite{MR2814399}.  Let us note here however that either $\mu(B)=0$ or $\mu(B)=1$ (an extension of the classical zero-one law, see \cite{MR1642391}).   

From now on we will not distinguish between a measure $\mu$ on $B$ and its unique extension to $H$ when it is clear which space we are considering.

\section{White Noise Representation}
\subsection{The General Case}
The setting is the following:  $B$ is a Banach space densely embedded in some Hilbert space $H$ (possibly with $B=H$), where $H$ is identified with its dual, $H=H^*$.  (A Hilbert space equal to its dual in this way is called a Pivot Space, see \cite{MR1782330}). 

	The classical definition of canonical representation has no immediate analogue for fields not indexed by $\mathbb R$, but the notion of strong representation does.  Let $L:H_\mu\to H$ be unitary.  Then $W_X(h)=R_X(L^*(h))$ defines an isonormal process on $H$ and $\sigma(\{W_X(h)\}_{h\in H})=\sigma(H_X)=\sigma(\{X(l)\}_{l\in B^*})$ where the last inequality follows from the density of $H$ in $B^*$.

We now state our representation theorem.

\begin{thm}Let $B$ be a Banach space, $\mu$ a Gaussian measure on $B$, and $C_\mu$ the covariance of $\mu$ on $B^*\times B^*$.  Then $\mu$ is the distribution of a random variable in $B$ given as a white noise integral of the form \[\tag{3.1}X(l)=\int h(l)dW.\] for some $h:B^*\to H$ and a Hilbert space $H$, where $h|_{H}$ is a Hilbert-Schmidt operator on $H$.  Moreover, the representation is strong in the following sense: $\sigma(\{W_X(h)\}_{h\in H})=\sigma(\{X(l)\}_{l\in B^*})$.

\end{thm}
\begin{proof}
$B\subset H=H^*$ as above.  Let $W_X$ be the isonormal process constructed above and $C_\mu(l,k)$ the covariance of $\mu$. 
Let $L$ be a unitary map from $H_\mu$ to $H$ and define the function $k_L(l):B^*\to H$ by \[k_L(l)\equiv L(C_\mu(l,\cdot)).\]

Consider the Gaussian random variable determined by \[Y(l)\equiv\int k _L(l)dW_X.\]  We have \[Cov(Y(l_1),Y(l_2))=\<k_L(l_1),k_L(l_2)\>_H=\<C_\mu(l_1,\cdot),C_\mu(l_2,\cdot)\>_{H_\mu}=C_\mu(l_1,l_2)\] so that $\mu$ is the distribution of $Y(l)$ and \[X(l)\stackrel{d}{=}\int k _L(l)dW_X.\]  

It is clear that $k_L$ is linear and if $C_\mu(h_1,h_2)=\<K(h_1),h_2\>_H$, $h_1,h_2\in H$, then from above \[k_L^*k_L=K.\] Because $K$ is trace class this implies that $k_L$ is Hilbert-Schmidt on $H$.

From the preceding discussion we have  $\sigma(\{W_X(h)\}_{h\in H})=\sigma(\{X(l)\}_{l\in B^*})$.
\end{proof}
\begin{rmk}  While the statement of the above theorem is more general than is needed for most applications, this generality serves to emphasize that having a ``factorable" covariance and thus an integral representation are basic properties of all Banach space valued Gaussian random variables.
\end{rmk}
\begin{rmk}  The kernel $h(l)$ is unique up to unitary equivalence on $H$, that is if $L'=UL$ for some unitary $U$ on $H$ $L$ as above, then \[\int h_{L'}(l)dW\stackrel{d}{=}\int U\left(h_L(l)\right)dW\stackrel{d}{=}\int h_L(l)dW.\]  

\end{rmk}
\begin{rmk}  In the proof above, \[\tag{3.2}\<k_L(l_1),k_L(l_2)\>_H= C_\mu(l_1,l_2)\]is essentially the ``canonical factorization" of the covariance operator given in \cite{MR1404659}, although in a slightly different form.\end{rmk}
\begin{rmk}In the language of stochastic partial differential equations, what we have shown is that every Gaussian random variable in a Hilbert space $H$ is the solution to the operator equation \[ L(X)=W\] for some closed unbounded operator $L$ on $H$ with inverse given by a Hilbert-Schmidt operator on $H$.
\end{rmk}
\subsection{Gaussian Random Fields}

The proof of Theorem 1 has the following corollary for Gaussian random fields:

\begin{cor}Let $X$ be a compact Hausdorff space, $\nu$ a positive Radon measure and $H=L^2(X,d\nu)$.  If $\{B_x\}$ is a collection of centered Gaussian random variables indexed by $X$, then $\{B_x\}$ has a version with sample paths belonging almost surely $H$ if and only if \[\tag{3.3}B_x\stackrel{d}{=}\int h(x,\cdot)dW\] for some $h:X\to H$ such that the operator $K(f)\equiv \int_X h(x,z)f(z)d\nu(z)$ is Hilbert-Schmidt.  In this case (3.2) takes the form \[\E[B_xB_y]=\int_Xh(x,z)h(y,z)d\nu(z).\]
\end{cor}

In other words, the field $B_x$ determines a Gaussian measure on $L^2(X,d\nu)$ if and only if $B_x$ admits an integral representation (3.3).

\subsection{Some Consequences and Examples}
In principle, all properties of a field are determined by its integral kernel.  Without making an exhaustive justification of this statement we give some examples:

In Corollary 2 above, being the kernel of a Hilbert-Schmidt operator, $h\in L^2(X\times X,d\nu\times d\nu)$.  This means that we can approximate $h$ by smooth kernels (supposing these are available).  If we assume $h(x,\cdot)$ is continuous as a map from $X$ to $H$ i.e. \[\lim_{x\to y}\|h(x,\cdot)-h(y,\cdot)\|_H=0\]  for each $y\in X$ and let $h_n\in C^\infty (X)$, $h_n\stackrel{L^2}{\to}h$ it follows that $\|h_n(x,\cdot)-h(x,\cdot)\|_H\to0$ pointwise so that if \[B^n_x=\int h_n(x,\cdot)dW\] we have \[\E[B^n_xB^n_y]\to\E[B_xB_y]\] point-wise.  This last condition is equivalent to \[B^n\stackrel{d}\to B\] and we can approximate in distribution any field over $X$ with a continuous (as above) kernel by fields with smooth kernels.

The kernel of a field over $\mathbb R^d$ describes its local structure \cite{MR1922445}: The limit in distribution of \[\lim_{\substack{r_n\to0\\ c_n\to0}}\frac{X{(t+c_nx)}-X(t)}{r_n}\] is \[\lim_{\substack{r_n\to0\\ c_n\to0}}\int \frac{h{(t+c_nx)}-h(t)}{r_n}dW\] where $h$ is the integral kernel of $X$, and this last limit is determined by the limit in $H$ of \[\lim_{\substack{r_n\to0\\ c_n\to0}}\frac{h{(t+c_nx)}-h(t)}{r_n}.\]

The representation theorem yields a simple proof of the known series expansion using the RKHS.  The setting is the same as in Theorem 1.

\begin{prop} Let $Y(l)$ be a centered Gaussian random variable in a Hilbert space $H$ with integral kernel $h(l)$.  Let $\{e_k\}_1^\infty$ be a basis for $H$.  Then there exist i.i.d. standard normal random variables $\{\xi_k\}$ such that 
\[Y(l)=\sum_1^\infty\xi_k\Phi_k(l)\] where $\Phi_k(l)=\<h(l),e_k\>_H$ and the series converges in $L^2(\Omega)$ and a.s.
\end{prop}
\begin{proof}For each $l$ \[h(l){=}\sum_1^\infty \Phi_k(l)e_k.\]  We have \[Y(l)=\int\sum_1^\infty \Phi_k(l)e_kdW=\sum_1^\infty \Phi_k(l)\xi_k\] where $\{\xi_k\}=\{\int e_kdW\}$ are i.i.d$.$ standard normal as $\int dW$ is unitary from $H$ to $L^2(\Omega)$.  As $\{\Phi_k(l)\}\in l^2(\mathbb N)$ the series converges a.s$.$ by the martingale convergence theorem.
\end{proof}

\section {Stochastic Integration}
Combined with Theorem 1 above, \cite{MR1972287} furnishes a theory of stochastic integration for Gaussian processes and fields, which we now describe for the case of a random field with square integrable sample paths as in Corollary 2.

	Denote by $\mu$ the distribution of $\{B_x\}$ in $H=L^2(X,d\nu)$ and as above the RKHS of $B_x$ by $H_\mu\subset H$.  Let \[B_x=\int h(x,\cdot)dW\]  and $L^*(f)=\int h(x,y)f(y)d\nu(y)$.  Then $L^*:H\to H_\mu$ is an isometry and the map $v\mapsto R_B(L^*(v))\equiv W(v): H\to H_B$ ($H_B$ is the closed linear span of $\{B_x\}$ as defined in sec$.$ 2) defines an isonormal process on $H$.  Denote this particular process by $W$ in what follows.  

First note that as $H_\mu=L^*(H)$ and $L$ is unitary, it follows immediately that $\mathbb D^{1,2}_{H_\mu}=L^*(\mathbb D^{1,2}_H)$ where we use the notation in \cite{MR2200233, MR1972287} and the subscript indicates the underlying Hilbert space.

The following proof from \cite{MR1972287} carries over directly:  For a smooth variable $F(h)=f(B(L^*(h_1),...,B(L^*(h_n))$ we have 
\begin{align}\notag \E\<D^B(F),u\>_{H_\mu}=&\E\<\sum_1^n f'(B(L^*(h_1),...,B(L^*(h_n))L^*(h_k),u\>_{H_\mu}\\\notag&=\E\<\sum f'(B(L^*(h_1),...,B(L^*(h_n))h_k,L(u)\>_H\\&\notag=\E\<\sum f'(W(h_1),...,W(h_n))h_k,L(u)\>_H\\&\notag=\E\<D^{W}(F),L(u)\>_H \end{align} which establishes 
\[dom(\delta^B)= L^*(dom(\delta^{W}))\] and \[\int udB=\int L(u)d{W}.\]

The series approximation in \cite{MR1972287} also extends directly to this setting.
\begin{thm}If $\{\Phi_k\}$ is a basis of $H_\mu$ then there exists i.i.d. standard normal $\{\xi_k\}$ such that: \hfill

\begin{enumerate}\item If $f\in H_\mu$ and \[\int fdB=\sum_1^\infty\<f,\Phi_k\>_{H_\mu}\xi_k\quad a.s.\]
\item If $u\in\mathbb D_{H_\mu}$ then \[\int udB=\sum_1^\infty (\<u,\Phi_k\>_{H_\mu}-\<D^B_{\Phi_k}u,\Phi_k\>_{H_\mu})\quad a.s.\]
\end{enumerate}

\end{thm}

\begin{proof}  The proof of (1) and (2) follows that in \cite{MR1972287}.

\end{proof}

\begin{rmk}  For our purposes the method of approximation via series expansions above seems most appropriate.  However in \cite{MR1849177} a Riemann sum approximation is given under certain regularity hypotheses on the integral kernel of the process, and this could be extended in various situations as well.
\end{rmk}
\begin{rmk}The availability of the kernel above suggests the method in \cite{MR1849177} whereby conditions are imposed on the kernel in order to prove an Ito Formula as promising for extension to more general settings.
\end{rmk}

\section{References}

\bibliographystyle{plain}
\bibliography{article}
\end{document}